\documentclass[leqno]{amsart}
\usepackage[T1]{fontenc}
\usepackage[utf8]{inputenc}

\usepackage{etoolbox}

\usepackage{amssymb}
\newcommand{\bm}[1]{\mathbf{#1}}
\newcommand{\R}[1]{\ifnumequal{#1}{1}{\mathbb{R}}{\mathbb{R}^{#1}}}

\newcommand{\ud}{\mathrm{d}}
\newcommand{\norm}[1]{\left\lVert#1\right\rVert}
\newcommand{\abs}[1]{\left\lvert#1\right\rvert}
\DeclareMathOperator{\Div}{div}
\allowdisplaybreaks

\newtheorem{theorem}{Theorem}
\newtheorem{lemma}{Lemma}[section]
\theoremstyle{definition}
\newtheorem{definition}[lemma]{Definition}
\newtheorem{remark}[lemma]{Remark}

\def\m@th{\mathsurround=0pt}
\def\eqal#1{\null\,\vcenter{\openup\jot\m@th
 \ialign{\strut\hfil$\displaystyle{##}$&&$\displaystyle{{}##}$\hfil
 \crcr#1\crcr}}\,}

\def\N{{\mathbb N}}

\def\rot{{\rm rot}\,}

\usepackage{tikz}
\usetikzlibrary{decorations.markings}

\numberwithin{equation}{section}

\begin{document}

\title[On weighted estimates for certain stream function]{On weighted estimates for the stream function of axially symmetric solutions to the Navier-Stokes equations in a bounded cylinder}

\author[B. Nowakowski]{Bernard Nowakowski}
\address{Military University of Technology\\Cybernetics Faculty\\Institute of Mathematics and Cryptology\\Warsaw\\Poland}
\email{bernard.nowakowski@wat.edu.pl}

\author[W.M. Zajączkowski]{Wojciech M. Zajączkowski}
\address{Military University of Technology\\Cybernetics Faculty\\Institute of Mathematics and Cryptology\\Warsaw\\Poland}
\address{Polish Academy of Sciences\\Institute of Mathematics\\Warsaw\\Poland}
\email{wz@impan.pl}



\begin{abstract}
    Higher-order estimates in weighted Sobolev spaces for solutions to a singular elliptic equation for the stream function in an axially symmetric cylinder are provided. These estimates are
    essential for investigating the existence of axially symmetric solutions to incompressible Navier-Stokes equations in axially symmetric cylinders. To derive the
    estimates the technique of Kondratiev is incorporated. The weight has a form of a power function of the distance to the axis of symmetry.
\end{abstract}

\maketitle

\section{Introduction}

In this note we derive estimates for solutions to the following problem
\begin{equation}\label{eq1.10}
    \left\{
        \begin{aligned}
            &-\Delta \psi + \frac{\psi}{r^2} = \omega\ & &\text{in $\Omega$},\\
            &\psi = 0 & &\text {on $S := \partial \Omega$},
        \end{aligned}
    \right.
\end{equation}
where $\Omega \subset \R3$ is a bounded cylinder with boundary $S$. Before we go into any geometrical details (see \eqref{eq1.50}) we briefly justify why this problem is highly important in mathematical fluid mechanics. 

Our ultimate goal is to study the regularity of weak solutions to an initial-boundary value problem to the three-dimensional axi-symmetric Navier-Stokes equations with a non-vanishing \emph{swirl}. In order to define this quantity we need to introduce cylindrical coordinates. If $x = (x_1,x_2,x_3)$ is in the Cartesian coordinates, then the cylindrical coordinates $(r, \varphi, z)$ are introduced by the relation $x = \boldsymbol \Phi(r,\varphi, z)$, where
\begin{equation*}
    \begin{aligned}
        x_1 &= r\cos\varphi,\\
        x_2 &= r\sin\varphi,\\
        x_3 &= z.
    \end{aligned}.
\end{equation*}
Thus, the standard basis vectors are
\begin{equation*}
    \begin{aligned}
        \bar e_r &= \partial_r \boldsymbol\Phi = (\cos\varphi,\sin\varphi,0).\\
        \bar e_\varphi & = \partial_{\varphi} \boldsymbol \Phi = (-\sin\varphi,\cos\varphi,0).\\
        \bar e_z &= \partial_z \boldsymbol\Phi = (0,0,1).
    \end{aligned}
\end{equation*}

Let $\bm w = \bm w(x, t)$ be any vector-valued function of $x$ and $t$. Then in cylindrical coordinates $\bm w$ is expressed in standard basis as follows 
\begin{equation}
    \bm w = w_r(r, \varphi, z, t)\bar e_r + w_\varphi(r,\varphi, z,t)\bar e_\varphi + w_z(r,\varphi, z,t)\bar e_z.
\label{1.3}
\end{equation}

We call $\bm w$ \emph{axially-symmetric} if
\[
    w_{r,\varphi} = w_{\varphi,\varphi} = w_{z,\varphi} = 0.
\]
In the mathematical theory of fluid mechanics we call function $r w_\varphi$ the \emph{swirl}. 

Let $\bm v$ and $p$ denote the velocity field of an incompressible fluid and the pressure, respectively. Let $\rot \bm v$ be the vorticity vector. Then, the Navier-Stokes equations read
\begin{equation}\label{1.4}
    \left\{
        \begin{aligned}
            &\bm v_t + (\bm v\cdot\nabla) \bm v - \nu\Delta \bm v + \nabla p = \bm f & &\text{in } \Omega^T = \Omega\times(0,T)\\
            &\Div \bm v = 0 & &\text{in }\Omega^T\\
            &\bm v \cdot \bar n = 0 & &\text{on } S^T = S\times(0,T)\\
            &\bm v \cdot \bar e_\varphi = 0 & &\text{on } S^T\\
            &\rot \bm v\cdot\bar e_\varphi=0 & &\text{on } S^T\\
            &\bm v\vert_{t=0} = \bm v_0 & &\text{in } \Omega
        \end{aligned}
    \right.,
\end{equation}
where $\bm f$ is the external force field and $\bar n$ is the unit outward vector normal to $S$ and $\Omega$ is the same domain as in \eqref{eq1.10}. 

The problem of regularity of axially-symmetric solutions to \eqref{1.4} in general is open. Since 1968 (see \cite{Ladyzenskaja:1968aa} and \cite{Ukhovskii:1968aa}) it is known that the Navier-Stokes equations have regular axially-symmetric solutions in $\R3$ provided that $v_\varphi\vert_{t = 0} = 0$ and $f_\varphi = 0$ (hence the swirl is zero). In case of non-vanishing swirl there are some partial results, e.g. \cite{Zhang:2014wk}, \cite{Wei:2016tn}, \cite{Chen:2017we}, \cite{Liu:2022tx}, though this list is far from complete. 

One way to investigate the existence of solutions to \eqref{1.4} is to start with the following observation: if $\bm v$ is axially symmetric solution to \eqref{1.4}, then in light of \eqref{1.3} we have 
\begin{equation*}
    \bm v = v_r(r,z,t)\bar e_r+v_\varphi(r,z,t)\bar e_\varphi+v_z(r,z,t)\bar e_z
\end{equation*}
and
\begin{equation*}
    \rot \bm v = -v_{\varphi,z}(r,z,t)\bar e_r+\omega(r,z,t)\bar e_\varphi+{1\over r}(rv_\varphi)_{,r}(r,z,t)\bar e_z
\end{equation*}
where
\begin{equation}\label{eq1.4}
    \omega = v_{r,z}-v_{z,r}.
\end{equation}
Expressing $(\ref{1.4})_2$ in the cylindrical coordinates yields
\begin{equation*}
    (rv_r)_{,r}+(rv_z)_{,z}=0
\end{equation*}
and combining this equation with \eqref{eq1.4} suggests introducing a stream function $\psi$ such that
\begin{equation}\label{eq1.5a}
    v_r =- \psi_{,z},\quad v_z={1\over r}(r\psi)_{,r}.
\end{equation}
Since
\[
    \Delta = \partial_r^2 + \frac{1}{r} \partial_r + \partial_z^2
\]
we see that this stream function satisfies \eqref{eq1.10}. Note that \eqref{eq1.10}$_2$ implies from \eqref{1.4}$_3$. This explains why \eqref{eq1.10} is of primary interest. Solutions to this problem are essential for establishing global, regular and axially-symmetric solutions to the Navier-Stokes equations with non-vanishing swirl. We will demonstrate this idea for the case of small swirl in \cite{NZ}. Having proper estimates for solutions to \eqref{eq1.10} the proof in \cite{NZ} is elementary.

There is a challenge in investigating \eqref{eq1.10}, which we shall now discuss. Let $a > 0$ and $R > 0$. Then a bounded cylinder $\Omega$ in cylindrical coordinates is given by
\begin{equation}\label{eq1.50}
    \Omega=\left\{x\in\R3\colon r<R,\ |z|<a\right\},
\end{equation}
where $S = \partial\Omega=S_1\cup S_2$ and
\[
    \begin{aligned}
        S_1 &= \left\{x\in\R3\colon r=R,\ |z|<a\right\}\\
        S_2 &= \left\{x\in\R3\colon r<R,\ z\in\{-a,a\}\right\}.
    \end{aligned}
\]

From the above description of $\Omega$ it follows that the terms $\frac{1}{r^2} \psi$ and $\frac{1}{r} \psi_r$ might be undefined for $r = 0$. This is a key challenge. There are a few possibilities for overcoming this issue: 
\begin{itemize}
    \item one could remove the $\epsilon$-neighborhood of $r = 0$, derive necessary estimates and pass with $\epsilon\to0^+$ (see e.g. \cite{Ladyzenskaja:1968aa}),
    \item consider $\frac{1}{r^{1-\epsilon}} \psi$, derive necessary estimates and pass with $\epsilon \to 0^+$ at the end (see e.g. \cite{Leonardi:1999uj}),
    \item use weighted Sobolev spaces.
\end{itemize}

We adopt the third approach. The classical results for the Poisson equation tell that if $\omega \in H^1$, then $\psi \in H^3$. We would expect a similar outcome but we need to handle $\frac{1}{r}$ and similar terms carefully. 

If we were interested in basic energy estimates we could proceed the standard way: multiply \eqref{eq1.10} by $\psi$, integrate by parts, use the H\"older and Cauchy inequalities. This would be justified because in light of \cite{LW} and \cite[Remark 2.4]{NZ2} we have
\begin{equation}\label{eq1.7a}
    \psi = O(r) \quad \text{as } r\to 0^+
\end{equation}
provided that $\psi$ is introduced through \eqref{eq1.5a} and $\bm v$ is an axially symmetric vector field of class $\mathcal{C}^1(0,R)$. Moreover, if $\bm v \in \mathcal{C}^3(0,R)$, then
\begin{equation}\label{eq1.8a}
    \psi = a_1(r,z,t)r + a_3(r,z,t)r^3 + o\left(r^4\right) \qquad \text{as } r\to 0^+,
\end{equation}
where $a_1$ and $a_3$ are smooth functions. Since basic energy estimates are not enough in our case, more sophisticated tools and techniques are needed. Weighted Sobolev spaces seem to be the right choice. 

To conduct our analysis we introduce the quantity $\psi_1=\psi/r$. We see that it satisfies
\begin{equation}\label{eq1.60}
    \left\{
        \begin{aligned}
            &-\Delta\psi_1 - \frac{2}{r}\psi_{1,r} = \frac{\omega}{r} \equiv\omega_1 & &\text{in $\Omega$}, \\
            &\psi_1=0 & &\text{on $S$}.
        \end{aligned}
    \right.
\end{equation}
Since $\psi = \frac{\psi}{r} r = \psi_1 r$ and $r$ is bounded by $R$ we see that any estimates for $\psi_1$ are immediately applicable to $\psi$. In fact, in \cite{NZ} we need estimates for $\psi_1$ because this function appears naturally in some auxiliary problems. 

To examine problem \eqref{eq1.60} in weighted Sobolev spaces we have to derive estimates with respect to $r$ and $z$, separately. To derive an estimate with respect to $r$ we have to examine solutions to (\ref{eq1.60}) independently as well in a neighborhood of the axis of symmetry as in a neighborhood located in a positive distance from it. To perform such considerations we treat $z$ as a parameter and we introduce a partition of unity $\{\zeta^{(1)}(r),\zeta^{(2)}(r)\}$ such that
\[
    \sum_{i=1}^2\zeta^{(i)}(r)=1
\]
and
\begin{equation*}
    \zeta^{(1)}(r) = 
    \begin{cases}
        1 & \text{for } r \leq r_0\\
        0 & \text{for } r \geq 2r_0
    \end{cases},
    \qquad\qquad
    \zeta^{(2)}(r) = 
    \begin{cases}
        0 & \text{for } r \leq r_0\\
        1 & \text{for } r \geq 2r_0
    \end{cases},
\end{equation*}
where $0 < r_0$ is fixed in such a way that $2r_0 < R$.

Let $\tilde\psi_1^{(i)}=\psi_1\zeta^{(i)}$, $i=1,2$ and $\dot\zeta = \frac{\ud}{\ud r}\zeta$, $\ddot\zeta = \frac{\ud^2}{\ud r^2}\zeta$. Then, from \eqref{eq1.60} we obtain two problems
\begin{equation}
    \left\{
    \begin{aligned}
        &-\Delta\tilde\psi_1^{(1)} - \frac{2}{r} \tilde\psi^{(1)}_{1,r} = \omega_1^{(1)} - 2\psi_{1,r}\dot\zeta^{(1)} - \psi_1\ddot\zeta^{(1)} - \frac{2}{r}\psi_1\dot\zeta^{(1)} & &\text{in $\Omega^{(1)}$}\\
        &\tilde\psi_1^{(1)} = 0 & &\text{on $\partial \Omega^{(1)}$}
    \end{aligned}
    \right.,
\label{1.12}
\end{equation}
where
\[
    \Omega^{(1)} = \left\{(r,z)\colon r > 0, z \in (-a,a)\right\}, \qquad \partial \Omega^{(1)} = \left\{(r,z)\colon z \in \{-a, a\}, r > 0\right\}
\]
and
\begin{equation}
    \left\{
    \begin{aligned}
        &-\Delta\tilde\psi_1^{(2)} - \frac{2}{r}\tilde\psi^{(2)}_{1,r} = \omega_1^{(2)} - 2\psi_{1,r}\dot\zeta^{(2)}-\psi_1\ddot\zeta^{(2)} - \frac{2}{r}\psi_1\dot\zeta^{(2)} & &\text{in $\Omega^{(2)}$}\\
        &\tilde\psi^{(2)} = 0 & &\text{on $\partial \Omega^{(2)}$}
    \end{aligned}
    \right.,
\label{1.13}
\end{equation}
where 
\begin{equation}\label{eq1.13a}
    \Omega^{(2)} = \left\{(r,z)\colon r_0 < r < R, z \in (-a,a)\right\}, \qquad \partial \Omega^{(2)} = \partial \Omega^{(2)}_1 \cup \partial \Omega^{(2)}_2
\end{equation}
and 
\[
    \partial \Omega^{(2)}_1 = \left\{(r,z)\colon z \in \{-a, a\}, r_0 < r < R\right\}, \quad \partial \Omega^{(2)}_2 = \left\{(r,z)\colon z \in (-a,a), r = R\right\}.
\]

We temporarily simplify the notation using
\begin{equation}\label{eq1.14a}
    \begin{aligned}
        u &= \tilde\psi_1^{(1)}, \qquad\qquad w = \tilde\psi_1^{(2)},\\
        f &= \omega_1^{(1)} - 2\psi_{1,r}\dot\zeta^{(1)} - \psi_1\ddot\zeta^{(1)} - \frac{2}{r}\psi_1\dot\zeta^{(1)},\\
        g &= \omega_1^{(2)} - 2\psi_{1,r}\dot\zeta^{(2)} - \psi_1\ddot\zeta^{(2)} - \frac{2}{r}\psi_1\dot\zeta^{(2)}.
    \end{aligned}
\end{equation}

Then \eqref{1.12} and \eqref{1.13} become
\begin{equation}
    \left\{
    \begin{aligned}\label{1.14}
        &-\Delta u - \frac{2}{r}u_{,r} = f & &\text{in $\Omega^{(1)}$}\\
        &u = 0 & &\text{on $\partial \Omega^{(1)}$}
    \end{aligned}
    \right.
\end{equation}
and 
\begin{equation}\label{1.15}
    \left\{
        \begin{aligned}
            &-\Delta w- \frac{2}{r} w_{,r} = g & &\text{in $\Omega^{(2)}$}\\
            &w = 0 & &\text{on  $\partial \Omega^{(2)}$}
        \end{aligned}
    \right..
\end{equation}

As we can see both above problems are similar. What differs them is the domain. In case of $\Omega^{(2)}$ we can safely use the classical theory for the Poisson equation. 

Since $r_0 > 0$ we instantly deduce that problem \eqref{1.15} can be solved classically. 

For studying the existence and properties of solutions to \eqref{1.14} we need the weighted Sobolev spaces. They are defined at the beginning of Section \ref{s2}. In addition we will be utilizing the Kondratiev technique (see \cite{Ko67}). It offers a way to deal with expression of the form $\frac{u}{r^{\alpha}}$ when $\alpha > 0$. We saw in \eqref{eq1.7a} that $\psi_1$ is well defined at $r = 0$ but in case of the weighted Sobolev space $H^3_0$ we would need to handle $\frac{\psi_1}{r^3}$ in $L_2$. Function $\psi_1$ does not have such an order of vanishing when $r \to 0^+$, thus it has to be modified in a certain way. These kinds of modifications form the essence of this note. 

The very first theorem we prove is the following:
\begin{theorem}\label{t1}
    Suppose that $\psi_1$ is a solution to (\ref{eq1.60}). Assume that $\omega_1\in L_{2,\mu}(\Omega)$, $\mu\in(0,1)$. Then the estimate holds
    \begin{multline*}
        \norm{\psi_1 - \psi_1(0)}_{L_2(-a,a;H_\mu^2(0,R))}^2 + \norm{\psi_{1,zr}}_{L_{2,\mu}(\Omega)}^2 + \norm{\psi_{1,zz}}_{L_{2,\mu}(\Omega)}^2 \\
+ 2\mu(2-2\mu) \norm{\psi_{1,z}}_{L_{2,\mu-1}(\Omega)}^2 \leq c\norm{\omega_1}_{L_{2,\mu}(\Omega)}^2,
\end{multline*}
where $\psi_1(0) = \psi_1\vert_{r=0}$.
\end{theorem}
In light of \eqref{eq1.8a} we cannot expect $\psi_1 \in H^2_\mu(0,R)$ for almost all $z$. However, this should be the case for the difference $\psi_1 - \psi_1\vert_{r = 0}$.

In a similar manner we obtain a higher order regularity

\begin{theorem}\label{t1.2}
    Let $\psi_1$ be a solution to \eqref{eq1.60}. Let $\omega_1\in H_\mu^1(\Omega)$, $\mu\in(0,1)$. Then
    \begin{multline*}
        \norm{\psi_1 - \psi_1(0)}_{L_2(-a,a;H_\mu^3(0,R))}^2 + \norm{\psi_{1,zzz}}_{L_{2,\mu}(\Omega)}^2 + \norm{\psi_{1,zzr}}_{L_{2,\mu}(\Omega)}^2 \\
        + 2\mu(2-2\mu)\norm{\psi_{1,zz}}_{L_{2,\mu-1}(\Omega)}^2\leq c\norm{\omega_1}_{H_\mu^1(\Omega)}^2.
    \end{multline*}
\end{theorem}

The above theorems are useful but we need the estimates when $\mu = 0$. We cannot simply pass with $\mu \to 0$ because $\psi_1 - \psi_1(0) \notin H^2_0$ nor $H^3_0$. Instead we construct two auxiliary functions $\chi$ and $\eta$ that we subtract from $\psi_1$ (this construction is presented in Lemmas \ref{l3.6} and \ref{lem3.7}). This allows us to derive necessary estimates in $H^3_0$. We emphasize that $H^3_0$ denotes a weighted Sobolev space (with the weight $\mu = 0$; see Section \ref{s2}) as opposed to a Sobolev space with zero traces. 

In the below theorems we assume that $\psi_1$ is a weak solution to \eqref{eq1.60}. Basic energy estimates and the existence of weak solutions are discussed in Section \ref{s2}.

\begin{theorem}\label{t1.3}
    Suppose that $\psi_1$ is a weak solution to \eqref{eq1.60}. Let $\omega_1\in L_2(\Omega)$ and introduce
    \[
        \chi(r,z) = \int_0^r\psi_{1,\tau}(1+K(\tau))\,\ud\tau,
    \]
    where $K(\tau)$ is a smooth function with a compact support such that
    \[
        \lim_{r\to 0^+} \frac{K(r)}{r^2} = c_0<\infty.
    \]
    Then 
    \begin{equation*}
        \norm{\psi_1 - \psi_1(0) - \chi}_{L_2(-a,a;H_0^2(0,R))}^2 + \norm{\psi_{1,zr}}_{L_2(\Omega)}^2 \\
        +\norm{\psi_{1,zz}}_{L_2(\Omega)}^2 \leq c\norm{\omega_1}_{L_2(\Omega)}^2,
    \end{equation*}
\end{theorem}

In case of $H^3_0$ we have
\begin{theorem}\label{t1.4}
    Let $\psi_1$ be a weak solution to (\ref{eq1.60}). Let $\omega_1\in H^1(\Omega)$. Then
    \begin{multline*}
        \int_{\mathbb{R}}\norm{\psi_1 - \psi_1(0) - \eta}_{H_0^3(\mathbb{R}_+)}^2\, \ud z + \int_{\mathbb{R}}\int_{\mathbb{R}_+} \left(\abs{\psi_{1,zzz}}^2 + \abs{\psi_{1,zzr}}^2 + \abs{\psi_{1,zz}}^2\right)\, r\ud r \ud z \\
        \leq c\norm{\omega_1}_{H^1(\Omega)}^2,
    \end{multline*}
    where 
    \[
        \eta(r,z) = - \int_0^r(r-\tau)\bigg({3\over r}\psi_{1,\tau}+\psi_{1,zz}+\omega_1\bigg) (1+K(\tau))\,\ud\tau
    \]
    and $K$ is the same as in Theorem \ref{t1.3}.
\end{theorem}

At this point the estimates from Theorems \ref{t1.3} and \ref{t1.4} may look surprising. In \cite{NZ} we show how to eliminate $\psi_1(0)$, $\chi$ and $\eta$ by the data. 

At the end of the Introduction it is worth mentioning that we could continue the process of deriving higher-order estimates for $\psi_1$. In light of \eqref{eq1.8a} it would require more subtractions from $\psi_1$ when $r = 0$. However, we do not see any potential gain nor immediate applications for such estimates. 

\section{Notation and auxiliary results}\label{s2}

\paragraph{Notation}

By $c$ we mean a generic constant which may vary from line to line. 

We also use $\mathbb{N} = \{1,2,\ldots\}$ and $\mathbb{N}_0 = \{0, 1, 2, \ldots \}$.

In is convenient to write: \emph{r.h.s.} -- the right-hand side and \emph{l.h.s.} -- the left-hand side.

The set $\{(r,z)\colon r > 0, z \in \R1\}$ we denote by $\R2_+$.

\paragraph{Function spaces}

\begin{definition}
Let $\Omega$ be either a cylindrical domain $(0,R) \times (-a, a)$ or $\Omega=\R2_+$. We introduce the following spaces
\[
    \begin{aligned}
        \|u\|^2_{L_{2,\mu}(\Omega)} &= \int_\Omega|u(r,z)|^2r^{2\mu}\, r\ud r\ud z,\ \ \mu\in\R1,\\
        \|u\|^2_{H_\mu^k(\Omega)} &= \sum_{|\alpha|\le k}\int_\Omega|D_{r,z}^\alpha u(r,z)|^2r^{2(\mu+|\alpha|-k)}\, r\ud r\ud z,
    \end{aligned}
\]
where $D_{r,z}^\alpha=\partial_r^{\alpha_1}\partial_z^{\alpha_2}$, $|\alpha|=\alpha_1+\alpha_2$, $|\alpha|\le k$, $\alpha_i\in\N_0$, $i=1,2$, $k\in\N_0$ and $\mu \in \R1$.
\end{definition}

Then the compatibility condition holds
$$
L_{2,\mu}(\Omega)=H_\mu^0(\Omega).
$$

\begin{remark}\label{r2.2}
For smooth functions with respect to $z$ we introduce the following weighted spaces
\begin{equation*}\label{2.1}
    \norm{u}_{H_\mu^k(\mathbb{R}_+)}^2 = \sum_{i = 0}^k \int_{\mathbb{R}_+}\abs{\partial_r^i u}^2 r^{2(\mu-k+i)}\, r\ud r
\end{equation*}
where $\mu\in\R1$ and $k\in\N_0$.

In view of transformation $\tau=-\ln r$, $r=e^{-\tau}$, $dr=-e^{-\tau}\,\ud\tau$ we have the equivalence
\begin{equation}\label{2.2}
    \sum_{i = 0}^k\int_{\R1_+} \abs{\partial_r^i u}^2 r^{2 (\mu - k + i)}\, r\ud r\sim \sum_{i = 0}^k\int_{\R1} \abs{\partial_\tau^iu'}^2e^{2 h \tau }\, \ud\tau
\end{equation}
which holds for $u'(\tau)=u'(-\ln r)=u(r)$, $h = k + 1 - \mu$.

We show equivalence (\ref{2.2}).

Take $k=0$. Then $h = 1 -\mu$ and
\[
    \int_\R1 \abs{u'}^2 e^{2(1 - \mu)\tau}\,\ud\tau = \int_{\R1_+} \abs{u(r)}^2 r^{2\mu - 2} \frac{1}{r}\, \ud r = \int_{\R1_+} \abs{u(r)}^2 r^{2\mu - 4}\,r\ud r.
\]
Take $k = 1$. Then $h= 2 -\mu$. Using that $\partial_\tau=-r\partial_r$ we have
\begin{multline*}
    \int_\R1\left(\abs{\partial_\tau u'}^2 + \abs{u'}^2\right) e^{2(1 - \mu)\tau}\,\ud\tau = \int_{\R1_+}\left(\abs{r\partial_ru}^2 + \abs{u}^2\right)r^{2\mu - 2}\frac{1}{r}\, \ud r \\
    = \int_{\R1_+}\left(\abs{\partial_ru}^2 + \frac{\abs{u}^2}{r^2}\right)r^{2\mu - 2}\,r\ud r.
\end{multline*}

Finally, take $k=2$. Then $h = 3 - \mu$ and $\partial_\tau^2=-r\partial_r(-r\partial_r)=r^2\partial_r^2+r\partial_r$. We have
\begin{multline*}
    \int_{\R1} \left(\abs{\partial_\tau^2 u}^2 + \abs{\partial_\tau u}^2 + \abs{u}^2\right)e^{2(1 - \mu)}\, \ud \tau \\
    = \int_{\R1_+} \left(\abs{r\partial_r(r \partial_r u)}^2 + \abs{r\partial_ru}^2 + \abs{u}^2\right)r^{2(\mu - 1)} \frac{1}{r}\, \ud r \\
    \leq \int_{\R1_+} \left( \abs{\partial_r^2 u}^2 + \frac{\abs{\partial_r u}^2}{r^2} + \frac{\abs{u}^2}{r^4}\right) r^{2\mu}\, r\ud r
\end{multline*}
and inversely
\begin{multline*}
    \int_{\R1_+} \left(\abs{\partial_r^2 u}^2 + \frac{\abs{\partial_r u}^2}{r^2} + \frac{\abs{u}^2}{r^4}\right) r^{2\mu}\, r \ud r\\
    = \int_{\R1} \left(\abs{e^{\tau} \partial_{\tau} (e^{\tau} \partial_{\tau} u')}^2 + \abs{e^{2\tau} \partial_{\tau} u'}^2 + e^{4\tau} \abs{u'}^2\right) e^{-2(\mu + 1)}\, \ud \tau \\
    \leq c \int_{\R1} \left(\abs{\partial_{\tau}^2 u'}^2 + \abs{\partial_{\tau} u'}^2 + \abs{u'}^2\right) e^{2(1 - \mu)}\, \ud \tau.
\end{multline*}

The above considerations imply equivalence (\ref{2.2}) for $k\le 2$. Similarly we prove equivalence (\ref{2.2}) for $k\ge 3$.

\paragraph{Fourier transform}

Let $f\in S(\R1)$, where $S(\R1)$ is the Schwartz space of all complex-valued rapidly decreasing infinitely differentiable functions on $\R1$. Then the Fourier transform and its inverse are defined by
\begin{equation}
\hat f(\lambda)={1\over\sqrt{2\pi}}\int_\R1e^{-i\lambda\tau}f(\tau)\,\ud\tau,\quad {\check{\hat f}}(\tau)={1\over\sqrt{2\pi}}\int_\R1e^{i\lambda\tau}\hat f(\lambda)d\lambda
\label{2.3}
\end{equation}
and $\check{\hat f}=\hat{\check f}=f$.

Using the Fourier transform we introduce equivalent norms to (\ref{2.1}) convenient for examining solutions of differential equations. Hence, by the Parseval identity we have
\begin{equation}
\int_{-\infty+ih}^{+\infty+ih}\sum_{j=0}^k|\lambda|^{2j}|\hat u(\lambda)|^2d\lambda=\int_\R1\sum_{j=0}^k|\partial_\tau^ju|^2e^{2h\tau}\,\ud\tau,
\label{2.4}
\end{equation}
where the r.h.s. norm is equivalent to norm (\ref{2.1}) under the equivalence (\ref{2.2}). This ends Remark \ref{r2.2}.
\end{remark}

\paragraph{Energy estimates and weak solutions}

\begin{lemma}\label{l2.3}
    Assume that $\omega_1\in L_2(\Omega)$. Then there exists a solution to problem (\ref{eq1.60}) such that $\psi_1\in H^1(\Omega)$ and the estimate holds
    \begin{equation}
    \|\psi_1\|_{H^1(\Omega)}^2+\int_{-a}^a\psi_1^2(0)dz\le c\|\omega_1\|_{L_2(\Omega)}^2,
    \label{2.5}
    \end{equation}
    where $\psi_1(0)=\psi_1|_{r=0}$.

    Moreover, we have also
    \begin{equation}\label{2.6}
        \norm{\psi_1}_{H^2(\Omega^{(2)})}\leq c\norm{\omega_1}_{L_2(\Omega))}
    \end{equation}
    where $\Omega^{(2)}$ was introduced in \eqref{eq1.13a}.
\end{lemma}

\begin{proof}
    Multiplying \eqref{eq1.60} by $\psi_1$, integrating over $\Omega$ and using the boundary condition and the Poincar\'e inequality we derive \eqref{2.5}. Then the existence follows from Fredholm alternative.

    To prove (\ref{2.6}) we use \eqref{1.15}
    \begin{equation*}
        \norm{-\Delta w}_{L_2(\Omega^{(2)})} \leq \norm{g}_{L_2(\Omega^{(2)})} + \norm{\frac{2}{r} w_{,r}}_{L_2(\Omega^{(2)})}.
    \end{equation*}
    In light of \eqref{1.13} and \eqref{eq1.14a} we obtain
    \begin{equation*}
        \norm{-\Delta w}_{L_2(\Omega^{(2)})} \leq c\left(\norm{\omega_1}_{L_2(\Omega)} + c \norm{w_{,r}}_{L_2(\Omega^{(2)})}\right).
    \end{equation*}
    Using \eqref{2.5} we conclude the proof.
    
\end{proof}

\begin{remark}\label{r2.4}
    We deduce from \eqref{eq1.8a} that 
    \begin{equation}
        \psi = a_1(z)r+z_2(z)r^3 + o(r^4) \qquad \text{when $r\to 0^+$.}
    \label{2.9}
    \end{equation}
    In particular $\psi(0) = \psi\vert_{r = 0} = 0$ but $\psi_1(0) = \psi_1\vert_{r = 0} \not=0$.
\end{remark}

\begin{lemma}\label{lem2.5}
    Assume that $\omega\in L_2(\Omega)$. Then there exists a solution $\psi\in H^1(\Omega)$ to problem \eqref{eq1.10} which satisfies 
    \begin{equation}
        \|\psi\|_{H^1(\Omega)}^2+\int_\Omega{\psi^2\over r^2}\, \ud x\le c\|\omega\|_{L_2(\Omega)}^2,
    \label{2.11}
    \end{equation}
    and
    \begin{equation}
        \norm{\psi_{,rz}}_{L_2(\Omega)}^2 + \norm{\psi_{,zz}}_{L_2(\Omega)}^2 + \int_\Omega\frac{\psi_{,z}^2}{r^2}\, \ud x \leq c \norm{\omega}_{L_2(\Omega)}^2.
    \label{2.12}
    \end{equation}
\end{lemma}

The proof of (\ref{2.11}) is similar to the proof of (\ref{2.5}). Moreover, in view of (\ref{2.9}) the second integral on the l.h.s. of (\ref{2.11}) is finite.

Now, we prove (\ref{2.12}).

\begin{proof}
    Multiply (\ref{eq1.10}) by $-\psi_{,zz}$ and integrate over $\Omega$ yields
    \begin{equation}
        \int_\Omega\psi_{,rr}\psi_{,zz}\, \ud x+\int_\Omega{1\over r}\psi_{,r}\psi_{,zz}\, \ud x+\int_\Omega\psi_{,zz}^2\, \ud x+\int_\Omega{\psi_{,z}^2\over r^2}\, \ud x \\
        =\int_\Omega\omega\psi_{,zz}\, \ud x
        \label{2.13}
    \end{equation}
    Integrating by parts in the first term and using the boundary conditions, we get
    \[
        \int_\Omega\psi_{,rr}\psi_{,zz}\, \ud x=\int_\Omega\psi_{,rz}^2\, \ud x+\int_\Omega \psi_{,rz}\psi_{,z}\, \ud r\ud z
    \]
    where the last term equals
    \[
        \int_{-a}^a \psi_{,z}^2\bigg\vert_{r=0}^{r=R} \, \ud z = 0
    \]
    because $\psi_{,z}|_{r=R}=0$ and (\ref{2.9}) holds.

    Similarly the second term in \eqref{2.13} vanishes. Hence, \eqref{2.13} implies \eqref{2.12} and concludes the proof.
\end{proof}

From \eqref{eq1.60} we derive the following problem

\begin{equation}\label{eq2.10aa}
    \left\{
        \begin{aligned}
            &-\Delta\psi_{1,z} - \frac{2}{r}\psi_{1,rz} = \omega_{1,z} & &\text{in $\Omega$}, \\
            &\psi_{1,z} = 0 & &\text{on $\{r = R, z \in (-a,a)\}$}, \\
            &\psi_{1,zz} = 0 & &\text{on $\{z \in \{-a,a\}, r < R\}$},
        \end{aligned}
    \right.
\end{equation}
where the last boundary condition follows from \eqref{eq1.60} and $\omega_1\big\vert_{z \in \{-a,a\}, r < R} = 0$.

\begin{lemma}
    Suppose that $\omega_{1,z} \in L_2(\Omega)$. Then there exists a weak solution to \eqref{eq2.10aa} such that $\psi_{1,z} \in H^1(\Omega)$ and
    \[
        \norm{\psi_{1,z}}^2_{H^1(\Omega)} + \int_{-a}^a \psi_{1,z}^2(0)\, \ud z \leq c \norm{\omega_1}^2_{L_2(\Omega)},
    \]
    where $\psi_{1,z}(0) = \psi_{1,z}\big\vert_{r = 0}$.
\end{lemma}

The proof is very similar to the proof of Lemma \ref{lem2.5}, thus will be omitted. 

\begin{lemma}[Hardy's inequalities]\label{lem2.6}
    From \cite[Appendix A]{Stein:1970wn} we have
    \[
        \left(\int_0^\infty \left(\int_0^x g(y)\, \ud y\right)^p x^{-r-1}\, \ud x\right)^{\frac{1}{p}} \leq \frac{p}{r}\left(\int_0^\infty \abs{y g(y)}^p y^{-r - 1}\, \ud y\right)^{\frac{1}{p}}
    \]
    for $g \geq 0$, $p \geq 1$ and $r > 0$.
\end{lemma}

\begin{remark}\label{rem2.8u}
    If we set $r = 1 - \alpha$ and $f(x) = \int_0^x g(y)\, \ud y$ in Lemma \ref{lem2.6} we obtain
    \[
        \int_0^\infty x^{\alpha-2}|f|^2\, \ud x\le{4\over(1-\alpha)^2}\int_0^\infty x^\alpha|f'(x)|^2\, \ud x,\ \ \alpha<1.
    \]
\end{remark}

\section{$L_2$-weighted estimates with respect to $r$ for solutions to (\ref{1.14})}

In this section we derive various estimates with respect to $r$ for solutions to \eqref{1.14} in the weighted Sobolev spaces using the technique of Kondratiev (see \cite{Ko67}). These estimates lay foundations for the proofs of Theorems \ref{t1}, \ref{t1.2}, \ref{t1.3} and \ref{t1.4}. The key idea is to treat variable $z$ as a parameter. 

First, we rewrite \eqref{1.14} in the form 
\begin{equation}
    \left\{
        \begin{aligned}
            &-u_{,rr} - \frac{3}{r} u_{,r} = f + u_{,zz} & &\text{in $\Omega^{(1)}$}, \\
            &u = 0 & &\text{on $\partial \Omega^{(1)}$}
        \end{aligned}
    \right.
\label{3.1}
\end{equation}

For a fixed $z \in (-a,a)$ we treat \eqref{3.1} as a 
\begin{equation}\label{eq3.2n}
    - u_{,rr} - \frac{3}{r} u_{,r} = f + u_{,zz} \qquad \text{in $\R1_+$}.
\end{equation}

Multiplying \eqref{3.1}$_1$ by $r^2$ we obtain
\[
    -r^2 u_{,rr} - 3ru_{,r} = r^2(f + u_{,zz}) \equiv g(r,z)
\]
or equivalently
\begin{equation}\label{eq3.3n}
    -r \partial_r(r \partial_r u) - 2r\partial_r u = g(r,z)
\end{equation}
Introduce the new variable
\[
    \tau = -\ln r, \qquad r = e^{-\tau}.
\]
Since $r \partial_r = - \partial_{\tau}$ we see that \eqref{eq3.3n} takes the form
\begin{equation}\label{eq3.4n}
    -\partial_\tau^2 u + 2\partial_\tau u = g\left(e^{-\tau}, z\right) \equiv g'(\tau, z)
\end{equation}
Utilizing the Fourier transform (see \eqref{2.3}) to \eqref{eq3.4n} we get
\begin{equation*}\label{eq3.5n}
    \lambda^2 \hat u + 2i\lambda \hat u = \hat g'.
\end{equation*}
For $\lambda \notin \{0, -2i\}$ we have 
\begin{equation}\label{eq3.6n}
    \hat u = \frac{1}{\lambda(\lambda + 2i)} \hat g' \equiv R(\lambda) \hat g'.
\end{equation}







\begin{lemma}\label{lem3.1}
    Assume that $f + u_{,zz} \in H^k_{\mu}(\R1_+)$, $k \in \mathbb{N}_0$, $\mu \in \R1$. Assume that $R(\lambda)$ does not have poles on the line $\Im \lambda = 1 + k - \mu$. Then, there exists a unique solution to \eqref{eq3.2n} in $H^{k + 2}_\mu(\R1_+)$ such that
    \begin{equation}\label{eq3.7n}
        \norm{u}_{H^{k + 2}_{\mu} (\R1_+)} \leq c \norm{f + u_{,zz}}_{H^k_{\mu}(\R1_+)}.
    \end{equation}
\end{lemma}

\begin{proof}
    Since $R(\lambda)$ does not have poles on the line $\Im \lambda = 1 + k - \mu = h$ we can integrate \eqref{eq3.6n} along the line $\Im \lambda = h$. Then we have
    \begin{multline}\label{eq3.8n}
        \int_{-\infty + ih}^{+\infty + ih} \sum_{j = 0}^{k + 2} \abs{\lambda}^{2(k + 2 - j)} \abs{\hat u}^2\, \ud \lambda \leq \int_{-\infty +ih}^{\infty + ih} \sum_{j = 0}^{k + 2} \abs{\lambda}^{2(k + 2 - j)} \abs{R(\lambda)\hat g'}^2\, \ud \lambda \\
        \leq c \int_{-\infty + ih}^{\infty + ih} \sum_{j = 0}^k \abs{\lambda}^{2(k - j)} \abs{\hat g'}^2\, \ud \lambda.
    \end{multline}

    By the Parseval identity (see \eqref{2.4}) inequality \eqref{eq3.8n} becomes
    \begin{equation*}
        \int_{\R1} \sum_{j = 0}^{k + 2} \abs{\partial_\tau^j u}^2 e^{2h\tau}\, \ud \tau \leq c \int_{\R1} \sum_{j = 0}^k \abs{\partial_\tau^j g'}^2 e^{2h\tau}\, \ud \tau.
    \end{equation*}
    Passing to variable $r$ yields
    \begin{equation*}
        \int_{\R1_+} \sum_{j = 0}^{k + 2} \abs{r^j \partial_r^j u}^2 r^{2(\mu - k - 1)}\frac{1}{r}\, \ud r \\
        \leq c\int_{\R1_+} \sum_{j = 0}^k \abs{r^j \partial_r^j g}^2 r^{2(\mu - k - 1)} \frac{1}{r}\, \ud r.
    \end{equation*}

    Continuing, we get
    \begin{equation*}
        \int_{\R1_+} \sum_{j = 0}^{k + 2} \abs{r^{j - (k + 2)} \partial_r^j u}^2 r^{2\mu}\, r \ud r \leq c \int_{\R1_+} \sum_{j = 0}^k \abs{r^{j - k} \partial_r^j (f + u_{,zz})}^2 r^{2\mu}\, r \ud r,
    \end{equation*}
    where the relation $g = r^2(f + u_{,zz})$ was used.
    
\end{proof}

\begin{remark}\label{rem3.2}
    Consider a solution $u$ to \eqref{eq3.2n}. In light of Lemma \ref{lem3.1} such a solution has certain regularity. Moreover, when we fix $\mu \in \R1$ we expect from $u$ certain behavior near $r = 0$. We are interested in two cases: $k = 0$ and $k = 1$. 

    When $k = 0$ we have $h = 1 - \mu$. Hence
    \begin{equation*}
        h_1 = 1 - \mu_1 < 0 \qquad \text{for $\mu_1 \in (1,2)$}
    \end{equation*}
    and
    \begin{equation*}
        h_2 = 1 - \mu_2 > 0 \qquad \text{for $\mu_2 \in (0,1)$}
    \end{equation*}

    Similarly, for $k = 1$ we have $h = 2 - \mu$ and
    \begin{equation*}
        \bar h_1 = 2 - \bar \mu_1 < 0 \qquad \text{for $\bar \mu_1 \in (2,3)$}
    \end{equation*}
    and
    \begin{equation*}
        \bar h_2 = 2 - \bar \mu_2 > 0 \qquad \text{for $\bar \mu_2 \in (0,2)$}
    \end{equation*}

    Function $R(\lambda)$ has a pole for $ h = \Im \lambda = 0$, thus the relations
    \begin{equation*}
        \begin{aligned}
            &1 - \mu_1 < 0 < 1 - \mu_2 \\
            &2 - \bar \mu_1 < 0 < 2 - \bar \mu_2
        \end{aligned}
    \end{equation*}
    hold. By Lemma \ref{lem3.1} we have four solutions:
    \[
        \begin{aligned}
            &k = 0: & &u_1 \in H^2_{\mu_1}(\R1_+), & &u_2 \in H^2_{\mu_2}(\R1_+) \\
            &k = 1: & &\bar u_1 \in H^3_{\bar \mu_1}(\R1_+), & &\bar u_2 \in H^3_{\bar \mu_2}(\R1_+)
        \end{aligned}
    \]
    Our aim is to investigate the relations between these solutions. 

\end{remark}

We will be using the notation from Remark \ref{rem3.2}.

\begin{lemma}
    Let $k = 0$. Then there exists a constant $c_0$ such that 
    \begin{equation}\label{eq3.16n}
        u_1 - u_2 = c_0.
    \end{equation}
    If $k = 1$, then we also have 
    \begin{equation}\label{eq3.17n}
        \bar u_1 - \bar u_2 = c_0.
    \end{equation}
\end{lemma}

\begin{proof}
    \begin{figure}
		\centering
		\begin{tikzpicture}
            \begin{scope}[decoration={markings, mark=between positions 0.1 and 1.9 step 2cm with {\arrow{<}}}]
    	    	\draw [thick, ->] (0,-2) -- (0,2) node[right, black] {$\Im \lambda$};
	    	    \draw [thick, ->] (-3,0) -- (3,0) node[below, black] {$\Re \lambda$};
                \draw[black, postaction={decorate}] (-2,-1) rectangle (2,1);
            \end{scope}
            \draw[thin, gray, dashed] (-3,1) -- (3,1) node[pos=0.5, above right] {$h_2$};
            \filldraw (0,1) circle (1.2pt);
            \draw[thin, gray, dashed] (-3,-1) -- (3,-1) node[pos=0.5, below right] {$h_1$};
            \filldraw (0,-1) circle (1.2pt);
            \filldraw (2,1) circle (1.5pt) node[above right] {$N+ih_2$};
            \filldraw (2,-1) circle (1.5pt) node[below right] {$N+ih_1$};
            \filldraw (-2,1) circle (1.5pt) node[above left] {$-N+ih_2$};
            \filldraw (-2,-1) circle (1.5pt) node[below left] {$-N+ih_1$};
            \filldraw (0,0) circle (1.2pt) node[above right] {$0$};
            \filldraw (0,-1.7) circle (1.2pt) node[right] {$-2i$};

		\end{tikzpicture}
        \caption{Line integral domain}\label{domain}
	\end{figure}
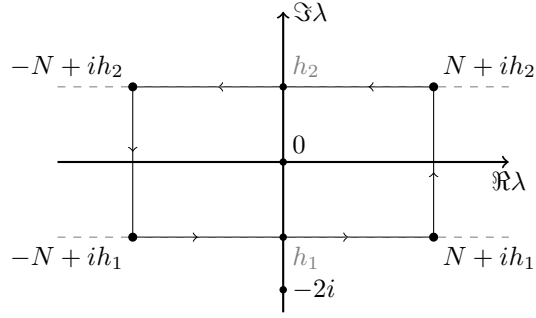

    Consider the case $k = 0$. Function $\hat g'$ is analytic for any $h \in (h_1, h_2)$ and
    \[
        \int_{-\infty + ih}^{+\infty + ih} \abs{\hat g'}^2\, \ud \lambda < \infty
    \]
    for any $h \in [h_1,h_2]$. We also have (see Fig. \ref{domain}) 
    \begin{multline*}
        u_1 = \lim_{N\to\infty} \int _{-N+ih_1}^{N + ih_1} e^{i\lambda \tau} \hat u(\lambda)\, \ud \lambda = \lim_{N\to \infty} \int_{-N+ih_1}^{N+ih_1} e^{i\lambda \tau} R(\lambda) \hat g'(\lambda)\, \ud \lambda \\
        = \operatorname{Res} e^{i\lambda \tau} R(\lambda) \hat g'(\lambda) - \lim_{N\to \infty} \left(\int_{N+ih_1}^{N+ih_2}e^{i\lambda \tau} R(\lambda) \hat g'(\lambda)\, \ud \lambda \right.\\
    \left. + \int_{N+ih_1}^{-N+ih_2}e^{i\lambda \tau} R(\lambda) \hat g'(\lambda)\, \ud \lambda + \int_{-N+ih_2}^{N+ih_2}e^{i\lambda \tau} R(\lambda) \hat g'(\lambda)\, \ud \lambda\right).
    \end{multline*}
    Passing with $N\to \infty$ yields
    \[
        u_1 = u_2 + \operatorname{Res} e^{i\lambda \tau} R(\lambda) \hat g'(\lambda) = u_2 + c_0,
    \]
    where
    \[
        u_j = \int_{-\infty + ih_j}^{\infty + ih_j} e^{i\lambda \tau} R(\lambda) \hat g'(\lambda)\, \ud \lambda.
    \]
    Hence \eqref{eq3.16n} holds.

    For $k=1$ operator $R(\lambda)$ has the same pole in the interval $(\bar h_1, \bar h_2)$. Hence \eqref{eq3.17n} holds. This ends the proof.
\end{proof}

\begin{remark}
    Let us compute $c_0$. Recall that $u_1 \in H^2_{\mu_1}(\R1_+)$ with $\mu_1 \in (1,2)$. It means that $u_1\big\vert_{r = 0} \neq 0$. But $u_2 = u_1 - c_0 \in H^2_{\mu_2}(\R1_+)$, $\mu_2 \in (0,1)$ so $u_2\big\vert_{r=0} = 0$. Hence
    \[
        c_0 = u_1(0) = u_1\big\vert_{r = 0}.
    \]

    Similarly, $\bar u_2 = \bar u_1 - c_0 \in H^3_{\bar \mu_2}(\R1_+)$ with $\bar \mu_2 \in (0,2)$ so
    \[
        c_0 = \bar u_1(0) \equiv \bar u_1\big\vert_{r = 0}.
    \]
    Investigating $\bar u_2 \in H^3_{\bar \mu_2}(\R1_+)$ with $\bar \mu_2 \in (0,1)$ we also need that
    \[
        \partial_r \bar u_2 = \partial_r \bar u_1 = 0 \qquad \text{for $r = 0$}.
    \]
    The restriction follows from Remark \ref{r2.4}.

\end{remark}

Functions $u_1$ and $\bar u_1$ are valid candidates for weak solutions to \eqref{eq3.2n} because they do not vanish on $r = 0$. 

Repeating the proof of Lemma \ref{lem2.5} we can show existence of weak solutions to \eqref{1.14} and the estimates
\begin{equation}\label{eq3.18n}
    \norm{u}_{H^1(\Omega^{(1)})}^2 + \int_{\Omega^{(1)}} \frac{u^2}{r^2}\, \ud x \leq c \norm{f}^2_{L_2(\Omega^{(1)})}
\end{equation}
and
\begin{equation}\label{eq3.19n}
    \norm{u_{,rz}}_{L_2(\Omega^{(1)})}^2 + \norm{u_{,zz}}^2_{L_2(\Omega^{(1)})} + \int_{\Omega^{(1)}} \frac{u_{,z}^2}{r^2}\, \ud x \leq c \norm{f}^2_{L_2(\Omega^{(1)})}
\end{equation}
Applying \eqref{eq3.7n} for $u = u_1$ and $\mu = \mu_1$ and using \eqref{eq3.19n} yields
\begin{equation*}\label{eq3.20n}
    \norm{u_1}^2_{L_2(-a,a;H^2_{\mu_1}(\R1_+))} \leq c \norm{f}^2_{L_2(-a,a;L_{2,\mu_1}(\R1_+))},
\end{equation*}
where $\mu_1 \in (1,2)$. The above inequality reflects increasing regularity of weak solutions to \eqref{1.14}.

Our aim is to find estimates in weighted Sobolev spaces for weak solutions to problem \eqref{1.14}. Let $u$ be such a weak solutions. We already know that $u$ satisfies \eqref{eq3.18n} and \eqref{eq3.19n}. Recalling properties of $u_1$ and $u_2$ and assuming that $f \in L_2(-a,a;L_{2,\mu}(\R1_+))$, $\mu \in (0,1)$ we can conclude that $u$ satisfies
\begin{equation*}\label{eq3.21n}
    \norm{u - u(0)}_{L_2(-a,a;H^2_{\mu}(\R1_+))} \leq c\norm{f}_{L_2(-a,a;L_{2,\mu}(\R1_+))}
\end{equation*}
where $u(0) = u\big\vert_{r = 0}$.

Recalling properties of $\bar u_1$ and $\bar u_2$ and assuming that
\[
    f + u_{,zz} \in L_2(-a,a;H^1_\mu(\R1_+))
\]
we conclude that 
\begin{equation}\label{eq3.22n}
    \norm{u - u(0)}_{L_2(-a,a;H^3_{\mu}(\R1_+))} \leq c \norm{f + u_{,zz}}_{L_2(-a,a;H^1_{\mu}(\R1_+))},
\end{equation}
where $\mu \in (0,1)$.

From \eqref{eq3.7n} for $k = 0$ and $\mu = 0$ we obtain for weak solutions to \eqref{1.14} the inequality
\begin{equation*}\label{eq3.23n}
    \norm{u - u(0)}_{L_2(-a,a;H^2_0(\R1_+))} \leq c \norm{f + u_{,zz}}_{L_2((-a,a)\times\R1_+)},
\end{equation*}

To derive estimate \eqref{eq3.22n} for $\mu = 0$ we see that $u - u(0)$ must be modified as it does not vanish quickly enough at $r = 0$. Thus, we introduce a new function $\eta(r,z)$ such that that $\left(u - u(0) - \eta(r,z)\right)_{,rr}\big\vert_{r=0} = 0$. Moreover, we would also need:

\begin{lemma}[cf. Lemma 4.12 in \cite{Ko67}]\label{lem3.5}
Let $\bar u\in H^k(\R1_+)$, $k\in\N$,\break
 ${\partial^i\over\partial r^i}u|_{r=0}=0$ for $i<k-1$ and $\partial_r^{k-1}\bar u\in H_0^1(\R1_+)$. Then $\bar u\in H_0^k(\R1_+)$ and
\begin{equation}
\|\bar u\|_{H_0^k(\R1_+)}\le c\|\partial_r^{k-1}\bar u\|_{H_0^1(\R1_+)}.
\label{3.55}
\end{equation}
\end{lemma}

\begin{proof}
Using the inequality
From Remark \ref{rem2.8u} we infer that
\[
    \int_0^\infty r^{-2}|\partial_r^{k-1}\bar u(r)|^2\,r\ud r\geq c\int_0^\infty r^{-4}|\partial_r^{k-2}\bar u(r)|^2\,r\ud r \geq c\int_0^\infty r^{-2k}|\bar u|^2\,r\ud r,
\]
which holds for $\partial_r^i\bar u|_{r=0}=0$, $i<k-1$. This implies (\ref{3.55}) and concludes the proof.
\end{proof}

Recall that $u$ is a solution to 
\begin{equation}\label{eq3.26u}
    u_{,rr} = - \left(\frac{3}{r} u_{,r} + u_{,zz} + f\right) \equiv g(r,z)
\end{equation}

\begin{lemma}\label{l3.6}
    Let $u$ solve \eqref{eq3.26u} and let $u\big\vert_{r = 0} = u(0)$. Assume that $u \in L_2(-a,a;H^3(\R1_+))$ and $f \in L_2(-a,a;H^1(\R1_+))$. Then there exists a function 
    \begin{equation}\label{eq3.27u}
        \eta(r, z) = \int_0^r (r - \tau) g(\tau, z)(1 + K(\tau))\, \ud \tau,
    \end{equation}
    where $K(r)$ is a smooth function with a compact support near $r = 0$ such that
    \[
        \lim_{r\to 0} K(r) r^{-2} = c_0 < \infty
    \]
    and the function
    \begin{equation}\label{eq3.28u}
        u - \eta - u(0) \in L_2(-a,a;H^3_0(\R1_+))
    \end{equation}
    satisfies the inequality
    \begin{multline}\label{eq3.29u}
        \norm{u - \eta - u(0)}_{L_2(-a,a;H^3_0(\R1_+))} \\
        \leq c\left( \norm{u}_{L_2(-a,a;H^2(\R1_+))} + \norm{f + u_{,zz}}_{L_2(-a,a;H^1(\R1_+))}\right).
    \end{multline}
\end{lemma}

\begin{proof}
    Since $u \in L_2(-a,a; H^3(\R1_+))$ we can work with $\mathcal{C}(-a,a;\mathcal{C}^{\infty}_0(\R1_+))$ and then use the density argument.

    We construct function $\eta$ as a solution to the equation
    \[
        \eta_{,rr} = g(r,z) (1 + K(r)).
    \]
    Integrating this equation we obtain \eqref{eq3.27u}. 

    To prove \eqref{eq3.28u} and \eqref{eq3.29u} we use Lemma \ref{lem3.5} for $k = 3$. To ensure its assumptions are met we check that
    \[
        \begin{aligned}
            \left(u - \eta - u(0)\right)\big\vert_{r = 0} &= -\eta\big\vert_{r = 0} = 0,\\
            \partial_r \left(u - \eta - u(0)\right)\big\vert_{r = 0} &= \partial_r(u - \eta)\big\vert_{r = 0} = \partial_r u \big\vert_{r = 0} - \partial_r \eta\big\vert_{r = 0} = 0,
        \end{aligned}
    \]
    where Remark \ref{r2.4} implies that $u_{,r}\big\vert_{r = 0}$ and 
    \[
        \partial_r \eta = \int_0^r g(\tau, z)(1 + K(\tau))\, \ud \tau
    \]
    gives $\partial_r \eta\big\vert_{r = 0} = 0$.

    Finally, we examine
    \begin{multline}\label{eq3.30u}
        \norm{\partial_{rr}(u - \eta - u(0))}_{H^1_0(\R1_+)} = \norm{\partial_{rr} uK}_{H^1_0(\R1_+)} \\
        = \norm{\left(\frac{3}{r}u_{,r} + u_{,zz} + f\right) K(r)}_{H^1_0(\R1_+)} \leq c \norm{u}_{H^2(\R1_+)} + \norm{f + u_{,zz}}_{H^1(\R1_+))}
    \end{multline}

    Applying Lemma \ref{lem3.5} and integrating \eqref{eq3.30u} with respect to $z$ we derive \eqref{eq3.28u} and \eqref{eq3.29u}. This ends the proof.
\end{proof}

\begin{lemma}\label{lem3.7}
    Let $u$ satisfy \eqref{eq3.26u}, $u\big\vert_{r = 0} = u(0)$, $u \in L_2(-a, a; H^2(\R1_+))$ and $f \in L_2(-a,a; L_2(\R1_+))$. Then, there exists a function
    \begin{equation}\label{eq3.31u}
        \chi(r,z) = \int_0^r u_{,\tau} (1 + K(\tau))\, \ud \tau,
    \end{equation}
    where $K$ is defined in Lemma \ref{l3.6} and the function
    \begin{equation}\label{eq3.32u}
        u - \chi - u(0) \in L_2(-a,a; H^2_0(\R1_+))
    \end{equation}
    satisfies
    \begin{equation}\label{eq3.33u}
        \norm{u - \chi - u(0)}_{L_2(-a,a;H^2_0(\R1_+))} \leq c\norm{u}_{L_2(-a,a;H^2(\R1_+))}.
    \end{equation}
\end{lemma}

\begin{proof}
    Since $u \in L_2(-a,a; H^2(\R1_+))$ we prove this lemma for functions from $\mathcal{C}(-a,a; \mathcal{C}^\infty_0(\R1_+))$ and use the density argument. 

    We construct function $\chi$ as a solution to
    \begin{equation}\label{eq3.34u}
        \chi_{,r} = u_{,r} (1 + K(r)).
    \end{equation}
    Integrating \eqref{eq3.34u} with respect to $r$ yields \eqref{eq3.31u}.

    To prove \eqref{eq3.32u} and \eqref{eq3.33u} we use Lemma \ref{lem3.5} for $k = 2$. We need to check its assumptions. We have
    \[
        (u - \chi - u(0))\big\vert_{r = 0} = (u - u(0))\big\vert_{r = 0} - \chi\big\vert_{r = 0} = 0
    \]
    and
    \begin{multline}\label{eq3.35u}
        \norm{(u - \chi - u(0))_{,rr}}_{H^1_0(\R1_+)} = \norm{\partial_r \int_0^r u_{,\tau}(\tau, z) K(\tau)\, \ud \tau}_{H^1_0(\R1_+)} \\
        = \norm{u_{,r}K + u K_{,r}}_{H^1_0(\R1_+)} + \norm{u K_{,r}}_{H^1_0(\R1_+)} \leq c\norm{u}_{H^2(\R1_+)}.
    \end{multline}
    Integrating \eqref{eq3.35u} with respect to $z$ and applying Lemma \ref{lem3.5} for $ k = 2$ we conclude the proof.
\end{proof}

Recall that $\psi_1$ is a solution to
\begin{equation}\label{eq3.36u}
    \left\{
    \begin{aligned}
        &-\psi_{1,rr} - \psi_{1,zz} - \frac{3}{r} \psi_{1,r} = \omega_1 & &\text{in $\Omega$},\\
        &\psi_1 = 0 & &\text{on $S_1 \cup S_2$}.
    \end{aligned}\right.
\end{equation}

\begin{lemma}\label{lem3.8}
    For solutions to \eqref{eq3.36u} the following estimates
    \begin{equation}\label{eq3.37u}
        \int_{\Omega} \left(\psi_{1,rr}^2 + \psi_{1,rz}^2 + \psi_{1,zz}^2\right)\, \ud x + \int_{\Omega} \frac{1}{r^2} \psi_{1,r}^2\, \ud x \leq c \norm{\omega_1}_{L_2(\Omega)}^2
    \end{equation}
    and
    \begin{equation}\label{eq3.38u}
        \int_{\Omega} \left(\psi_{1,zzr}^2 + \psi_{1,zzz}^2\right)\, \ud x \leq c \norm{\omega_{1,z}}_{L_2(\Omega)}^2
    \end{equation}
    hold.
\end{lemma}

\begin{proof}
    First we show \eqref{eq3.37u}. Multiplying \eqref{eq3.36u} by $\psi_{1,zz}$ and integrating over $\Omega$ yields
    \begin{equation}\label{eq3.39u}
        -\int_\Omega \psi_{1,rr} \psi_{1,zz}\, \ud x - \int_\Omega \psi_{1,zz}^2\, \ud x - 3\int_\Omega \frac{1}{r} \psi_{1,r}\psi_{1,zz}\, \ud x = \int_\Omega \omega_1 \psi_{1,zz}\, \ud x.
    \end{equation}
    The first term in \eqref{eq3.39u} equals
    \begin{multline*}
        -\int_\Omega(\psi_{1,rr}\psi_{1,z})_{,z}\, \ud x + \int_\Omega \psi_{1,rrz}\psi_{1,z}\, \ud x \\
        = -\int_\Omega (\psi_{1,rr}\psi_{1,z})_{,z}\, \ud x + \int_\Omega (\psi_{1,rz}\psi_{1,z})_{,r}\, \ud x - \int_\Omega \psi_{1,rz}^2\, \ud x,
    \end{multline*}
    where the first term is equal to
    \[
        -\int_0^R \psi_{1,rr}\psi_{1,z}\big\vert_{S_2}\, r\ud r = 0,
    \]
    because $\psi_{1,rr}\vert_{S_2} = 0$ and the second
    \[
        \int_{-a}^a \psi_{1,rz}\psi_{1,z}\big\vert_{S_1}\, \ud z = 0,
    \]
    which follows from $\psi_{1,z}\vert_{S_1} = 0$.

    Consider the last term on the l.h.s. of \eqref{eq3.39u}. We have
    \begin{multline*}
        -3 \int_\Omega \psi_{1,r}\psi_{1,zz}\, \ud r\, \ud z = -3 \int_\Omega (\psi_{1,r}\psi_{1,z})_{,z}\,\ud r\, \ud z + 3\int_\Omega \psi_{1,rz}\psi_{1,z}\, \ud x \\
        = -\frac{3}{2} \int_{-a}^a \psi_{1,z}^2\big\vert_{r = 0}^{r = R}\, \ud z,
    \end{multline*}
    where we used that 
    \[
        \int_0^R \psi_{1,r}\psi_{1,z}\big\vert_{S_2}\, \ud r = 0
    \]
    because $\psi_{1,r}\big\vert_{S_2} = 0$.

    Using the above considerations in \eqref{eq3.39u} implies
    \begin{equation}\label{eq3.40u}
        -\int_\Omega \left(\psi_{1,rz}^2 + \psi_{1,zz}^2\right)\, \ud x + \frac{3}{2} \int_{-a}^a \psi_{1,z}^2\big\vert_{r = 0}^{r = R}\, \ud z = \int_\Omega \omega_1 \psi_{1,zz}\, \ud x.
    \end{equation}
    Since $\psi_{1,z}\big\vert_{r=R} = 0$ equality \eqref{eq3.40u} can be written in the form
    \begin{equation}\label{eq3.41u}
        \int_\Omega (\psi_{1,rz}^2 + \psi_{1,zz}^2)\, \ud x + \frac{3}{2} \int_{-a}^a \psi_{1,z}^2\big\vert_{r = 0}\, \ud z = -\int_\Omega \omega_1 \psi_{1,zz}\, \ud x.
    \end{equation}
    Applying the H\"older and Young inequalities to the r.h.s of \eqref{eq3.41u} we obtain
    \begin{equation}\label{eq3.42u}
        \int_\Omega (\psi_{1,rz}^2 + \psi_{1,zz}^2)\, \ud x + \int_{-a}^a \psi_{1,z}^2\big\vert_{r = 0}\, \ud z \leq \int_\Omega \omega_1^2\, \ud x.
    \end{equation}

    Multiplying \eqref{eq3.36u} by $\frac{1}{r} \psi_{1,r}$ and integrating over $\Omega$ yields
    \begin{equation}\label{eq3.43u}
        3\int_\Omega \abs{\frac{1}{r} \psi_{1,r}}^2\, \ud x = -\int_\Omega \psi_{1,rr} \frac{1}{r} \psi_{1,r}\, \ud x - \int_\Omega \psi_{1,zz} \frac{1}{r} \psi_{1,r}\, \ud x - \int_\Omega \omega_1 \frac{1}{r} \psi_{1,r}\, \ud x.
    \end{equation}
    The first term on the r.h.s of \eqref{eq3.43u} equals 
    \[
        -\int_\Omega \psi_{1,r} \psi_{1,rr}\, \ud r\, \ud z = -\frac{1}{2} \int_\Omega \partial_r (\psi_{1,r}^2)\, \ud r\, \ud z = -\frac{1}{2} \int_{-a}^a \psi_{1,r}^2\big\vert_{r = R}\, \ud z
    \]
    because $\psi_{1,r}\big\vert_{r = 0} = 0$ (see Remark \ref{r2.4}).
    The second term on the r.h.s of \eqref{eq3.43u} reads
    \begin{multline*}
        -\int_\Omega \psi_{1,zz}\psi_{1,r}\, \ud r\,\ud z = -\int_\Omega (\psi_{1,z} \psi_{1,r})_{,z}\ud r\, \ud z + \int_\Omega \psi_{1,z} \psi_{1,rz}\, \ud r\, \ud z \\
        = -\int_{S_2} \psi_{1,z}\psi_{1,r}\, \ud r + \frac{1}{2} \int_\Omega \partial_r (\psi_{1,z}^2)\, \ud r\, \ud z,
    \end{multline*}
    where the first integral vanishes because $\psi_{1,r}\big\vert_{S_2} = 0$ and the second equals
    \[
        \frac{1}{2}\int_{-a}^a \psi_{1,z}^2\big\vert_{r = 0}^{r = R}\, \ud z = - \frac{1}{2}\int_{-a}^a \psi_{1,z}(0,z)\, \ud z
    \]
    because $\psi_{1,z}\big\vert_{r = R} = 0$.

    Using the above results in \eqref{eq3.43u} yields
    \begin{multline}\label{eq3.44u}
        3\int_\Omega \abs{\frac{1}{r}\psi_{1,r}}^2\, \ud x + \frac{1}{2} \int_{-a}^a \psi_{1,r}^2(R,z)\, \ud z + \frac{1}{2} \int_{-a}^a \psi_{1,z}^2 (0,z)\, \ud z \\
        = -\int_\Omega \omega_1 \frac{1}{r}\psi_{1,r}\, \ud x.
    \end{multline}
    Applying the H\"older and Young inequalities to the r..s of \eqref{eq3.44u} we obtain
    \begin{equation*}\label{eq3.45u}
        \int_\Omega \abs{\frac{1}{r} \psi_{1,r}}^2\, \ud x + \int_{-a}^a (\psi_{1,r}^2(R,z) + \psi_{1,z}(0,z))\, \ud z \leq c \int_\Omega \omega_1^2\, \ud x.
    \end{equation*}
    From \eqref{eq3.36u} we infer that
    \begin{equation*}\label{eq3.46u}
        \norm{\psi_{1,rr}}^2_{L_2(\Omega)} \leq \norm{\psi_{1,zz}}^2_{L_2(\Omega)} + 3 \norm{\frac{1}{r} \psi_{1,r}}^2_{L_2(\Omega)} + \norm{\omega_1}^2_{L_2(\Omega)}.
    \end{equation*}
    Combining the above inequalit with \eqref{eq3.42u} yields \eqref{eq3.37u}.

    Next we show \eqref{eq3.38u}. Differentiating \eqref{eq3.36u} with respect to $z$, multiplying by $-\psi_{1,zzz}$ and integrating over $\Omega$ we obtain
    \begin{multline}\label{eq3.47u}
        \int_\Omega \psi_{1,rrz} \psi_{1,zzz}\, \ud x + \int_\Omega \psi_{1,zzz}^2\, \ud x + 3 \int_\Omega \frac{1}{r} \psi_{1,rz}\psi_{1,zzz}\, \ud x \\
        = -\int_\Omega \omega_{1,z}\psi_{1,zzz}\, \ud x.
    \end{multline}
    Integrating by parts in the first term yields
    \begin{multline*}
        \int_\Omega (\psi_{1,rrz}\psi_{1,zz})_{,z}\, \ud x - \int_\Omega \psi_{1,rrzz}\psi_{1,zz}\, \ud x \\
        = \int_0^R \psi_{1,rrz} \psi_{1,zz}\big\vert_{z = -a}^{z = a}\, r\ud r - \int_\Omega \psi_{1,rzz}\psi_{1,zz}\,\ud r\, \ud z + \int_\Omega \psi_{1,rzz}^2\, \ud x \\
        + \int_\Omega \psi_{1,rzz}\psi_{1,zz}\, \ud r\, \ud z = \int_0^R \psi_{1,rrz}\psi_{1,zz}\big\vert_{z = -a}^{z = a}\, r \ud r - \int_{-a}^a \psi_{1,rzz}\psi_{1,zz}r\big\vert_{r = 0}^{r = R}\, \ud z \\
        + \int_\Omega \psi_{1,rzz}^2\, \ud x + \int_\Omega \psi_{1,rzz}\psi_{1,zz}\, \ud r\, \ud z \equiv I.
    \end{multline*}
    Since $\psi_{1,zz}\big\vert_{r = R} = 0$ and $\psi_{1,rzz}\big\vert_{r = 0} = 0$ the second term in $I$ vanishes. 
    To examine the first termin in $I$ we project \eqref{eq3.36u} onto $S_2$. Then we have
    \[
        \psi_{1,zz}\big\vert_{S_2} = -\psi_{1,rr}\big\vert_{S_2} - \frac{3}{r} \psi_{1,r}\big\vert_{S_2} - \omega_1\big\vert_{S_2}.
    \]
    Since $\omega_1\big\vert_{S_2} = 0$ and $\psi_1\big\vert_{S_2} = 0$ it follows that $\psi_{1,zz}\big\vert_{S_2} = 0$. Therefore $I$ becomes
    \[
        I = \int_\Omega \psi_{1,rzz}^2\, \ud x + \int_\Omega \psi_{1,rzz}\psi_{1,zz}\, \ud r\, \ud z.
    \]
    The second termin in $I$ is equal to
    \[
        -\frac{1}{2} \int_{-a}^a \psi_{1,zz}^2\big\vert_{r = 0}\, \ud z,
    \]
    where it is used that $\psi_{1,zz}\big\vert_{r = R} = 0$.

    The last term on the l.h.s of \eqref{eq3.47u} equals
    \begin{equation*}
        -3\int_\Omega \psi_{1,rzz}\psi_{1,zz}\, \ud r\ud z = - \frac{3}{2} \int_{-a}^a \psi_{1,zz}^2\big\vert_{r = 0}^{r = R}\, \ud z = \frac{3}{2} \int_{-a}^a \psi_{1,zz}^2\big\vert_{r = 0}\, \ud z,
    \end{equation*}
    where we used that $\psi_{1,zz}\big\vert_{S_2} = 0$ and $\psi_{1,zz}\big\vert_{r = R} = 0$.

    In view of the above calculations equality \eqref{eq3.47u} takes the form
    \begin{equation}\label{eq3.48u}
        \int_\Omega (\psi_{1,rzz}^2 + \psi_{1,zzz}^2)\, \ud x + \int_{-a}^a \psi_{1,zz}^2\, \ud z = -\int_\Omega \omega_{1,z}\psi_{1,zzz}\, \ud x
    \end{equation}
    Applying the H\"older and Young inequalities to the r.h.s of \eqref{eq3.48u} gives
    \begin{equation*}\label{eq3.49u}
        \int_\Omega (\psi_{1,rzz}^2 + \psi_{1,zzz}^2)\, \ud x + \int_{-a}^a\psi_{1,zz}^2\big\vert_{r = 0}\, \ud z \leq \int_\Omega \abs{\omega_{1,z}}^2\, \ud x.
    \end{equation*}
    The above inequality implies \eqref{eq3.38u} and concludes the proof.

\end{proof}

\begin{remark}
    Up to now we have considered problem (\ref{3.1}) treating $z$ as a parameter. It describes solutions to (\ref{eq1.60}) only in a neighborhood of the axis of symmetry. Solutions to (\ref{eq1.60}) in a domain $r > r_0 > 0$ are described by problem (\ref{1.15}). 
    From \eqref{2.5}, \eqref{eq3.37u} and \eqref{eq1.14a}$_3$ we obtain for solutions to \eqref{1.15} the estimate
    \begin{equation}\label{eq3.50u}
        \norm{\omega}_{H^{2 + k}(\Omega^{(2)})} \leq c \norm{\omega_1}_{H^k(\Omega^{(2)})},
    \end{equation}
    where $k \in \{0,1\}$. Since $\operatorname{supp} \omega \subset \Omega^{(2)}$ we see that \eqref{eq3.50u} can be deduced for weighted spaces
    \begin{equation*}\label{eq3.51u}
        \norm{w}_{H^{2 + k}_\mu(\Omega^{(2)})} \leq c \norm{\omega_1}_{H^k_\mu(\Omega^{(2)})}, \qquad \mu \geq 0.
    \end{equation*}
\end{remark}

\section{Estimates with respect to $z$ for solutions to (\ref{eq1.60})}

Consider problem (\ref{eq1.60}) in the form
\begin{equation}
    \left\{
    \begin{aligned}
        &-\psi_{1,rr} - \frac{3}{r}\psi_{1,r} - \psi_{1,zz} = \omega_1 & &\text{in $\Omega$},\\
        &u = 0 & &\text{for $z\in\{-a,a\}$}, \\
        &u = 0 & &\text{for $r=R$}.
    \end{aligned}
    \right.
\label{4.1}
\end{equation}

\begin{lemma}\label{lem4.1}
    Fix $\mu \in [0,1)$. Assume that $\omega_1\in L_{2,\mu}(\Omega)$. Then the following estimate holds
    \begin{equation}
        \int_\Omega\left(\psi_{1,zz}^2 + \psi_{1,zr}^2\right)r^{2\mu}\, \ud x + 2\mu(1 - \mu)\int_\Omega \psi_{1,z}^2 r^{2\mu - 2}\, \ud x \leq c\int_\Omega \omega_1^2 r^{2\mu}\, \ud x.
    \label{4.2}
    \end{equation}
\end{lemma}

\begin{proof}
    Multiply $(\ref{4.1})_1$ by $-\psi_{1,zz}r^{2\mu}$ and integrate over $\Omega$. Then we have
    \begin{multline}
        \int_\Omega \psi_{1,zz}^2 r^{2\mu}\, \ud x + \int_\Omega \psi_{1,rr} \psi_{1,zz}r^{2\mu}\, \ud x + 3\int_\Omega \frac{1}{r}\psi_{1,r}\psi_{1,zz}r^{2\mu}\, \ud x\\
        =-\int_\Omega \omega_1 \psi_{1,zz}r^{2\mu}\, \ud x.
    \label{4.3}
    \end{multline}
    Integrating by parts in the second term on the l.h.s. we obtain
    \begin{multline*}
        -\int_\Omega \psi_{1,rrz} \psi_{1,z}r^{2\mu}\, \ud x = -\int_\Omega \psi_{1,rrz} \psi_{1,z}r^{2\mu+1}\, \ud r\ud z = -\int_\Omega(\psi_{1,rz} \psi_{1,z}r^{2\mu+1})_{,r}\, \ud r\ud z \\
        + \int_\Omega \psi_{1,rz}^2r^{2\mu}\, \ud x +(2\mu+1)\int_\Omega \psi_{1,rz}\psi_{1,z}r^{2\mu}\, \ud r\ud z\equiv I_1+I_2+I_3.
    \end{multline*}
    We easily see that
    \[
        I_1=-\int_{-a}^a\psi_{1,rz}\psi_{1,z}r^{2\mu}\bigg\vert_{r=0}^{r=R}\, \ud z  = 0
    \]
    because $\psi_{1,z}\big\vert_{r=R}=0$ and Remark \ref{r2.4} imply that $\psi_{1,z}\big\vert_{r=0}=0$.

    Finally
    \begin{multline*}
        I_3 = \frac{2\mu+1}{2}\int_\Omega\partial_r \psi_{1,z}^2r^{2\mu}\, \ud r\ud z = \frac{2\mu+1}{2}\int_\Omega\partial_r (\psi_{1,z}^2r^{2\mu})\, \ud r\ud z\\
    - \mu(2\mu+1)\int_\Omega \psi_{1,z}^2r^{2\mu-1}\, \ud r\ud z,
    \end{multline*}
    where the first integral vanishes under the same arguments used for $I_1$.

    The last term on the l.h.s. of (\ref{4.3}) equals
    \begin{multline*}
        -3\int_\Omega \psi_{1,zr}\psi_{1,z}r^{2\mu}\, \ud r\ud z = -\frac{3}{2}\int_\Omega\partial_r(\psi_{1,z}^2)r^{2\mu}\, \ud r\ud z \\
        =-\frac{3}{2}\int_\Omega\partial_r(\psi_{1,z}^2r^{2\mu})\, \ud r\ud z + 3\mu\int_\Omega \psi_{1,z}^2r^{2\mu-1}\, \ud r\ud z,
    \end{multline*}
    where the first integral vanishes by the same arguments as in the case of $I_1$.

    Using the above results in (\ref{4.3}) yields
    \begin{equation*}
        \int_\Omega\left(\psi_{1,zz}^2 + \psi_{1,zr}^2\right)r^{2\mu}\, \ud x + 2\mu(1-\mu)\int_\Omega \psi_{1,z}^2r^{2\mu-2}\, \ud x \leq c\int_\Omega\omega_1^2r^{2\mu}\, \ud x.
        \label{4.4}
    \end{equation*}
    The above inequality implies (\ref{4.2}) and concludes the proof.
\end{proof}

\begin{lemma}\label{lem4.2}
    Fix $\mu \in [0,1)$. Assume that $\omega_{1,z}\in L_{2,\mu}(\Omega)$. Then
    \begin{equation}
        \int_\Omega\left(\psi_{1,zzz}^2+\psi_{1,rzz}^2\right)r^{2\mu}\, \ud x+2\mu(1-\mu)\int_\Omega \psi_{1,zz}^2r^{2\mu-2}\, \ud x \\
        \leq c\int_\Omega \omega_{1,z}^2r^{2\mu}\, \ud x.
    \label{4.5}
    \end{equation}
\end{lemma}

\begin{proof}
    Differentiate (\ref{4.1}) with respect to $z$, multiply by $-\psi_{1,zzz}r^{2\mu}$ and integrate over $\Omega$. Then we obtain
    \begin{multline}
        \int_\Omega \psi_{1,rrz}\psi_{1,zzz}r^{2\mu}\, \ud x+\int_\Omega \psi_{1,zzz}^2r^{2\mu}\, \ud x+3\int_\Omega{1\over r}\psi_{1,rz}\psi_{1,zzz}r^{2\mu}\, \ud x \\
        =-\int_\Omega \omega_{1,z}\psi_{1,zzz}r^{2\mu}\, \ud x.
    \label{4.6}
    \end{multline}
    From (\ref{4.1}) it follows that
    \begin{equation}
        \psi_{1,zz}\big\vert_{z\in\{-a,a\}} = 0
        \label{4.7}
    \end{equation}
    because $\psi_1\big\vert_{z\in\{-a,a\}}=0$ and $\omega_1\big\vert_{z\in\{-a,a\}}=0$.

    In view of (\ref{4.7}) the first integral on the l.h.s. of (\ref{4.6}) equals
    \begin{multline}
        -\int_\Omega \psi_{1,rrzz} \psi_{1,zz}r^{2\mu}\, \ud x = -\int_\Omega\left(\psi_{1,rzz} \psi_{1,zz}r^{2\mu+1}\right)_{,r}\, \ud r\ud z \\
        +\int_\Omega \psi_{1,rzz}^2 r^{2\mu}\, \ud x + (2\mu+1)\int_\Omega \psi_{1,rzz} \psi_{1,zz}r^{2\mu}\, \ud r\ud z.
        \label{4.8}
    \end{multline}
    In virtue of boundary condition $u|_{r=R}=0$ and Remark \ref{r2.4} the first integral on the r.h.s. of (\ref{4.8}) vanishes.
    
    Integrating by parts in the last term on the l.h.s. of (\ref{4.6}) and using (\ref{4.8}), we obtain
    \begin{multline}
        \int_\Omega\left(\psi_{1,zzz}^2 + \psi_{1,rzz}^2\right)r^{2\mu}\, \ud x + (2\mu-2)\int_\Omega \psi_{1,rzz} \psi_{1,zz}r^{2\mu}\, \ud r\ud z \\
        =-\int_\Omega \omega_{1,z}\psi_{1,zzz}r^{2\mu}\, \ud x.
    \label{4.9}
    \end{multline}
    The second term on the l.h.s. equals
    \begin{multline}
        (\mu-1)\int_\Omega\partial_r\left(\psi_{1,zz}^2\right)r^{2\mu}\, \ud r\ud z \\
        =(\mu-1)\int_\Omega \partial_r\left(\psi_{1,zz}^2r^{2\mu}\right)\, \ud r\ud z +2\mu(1-\mu)\int_\Omega \psi_{1,zz}^2r^{2\mu-1}\, \ud r\ud z \\
        =(\mu-1)\int_{-a}^a \psi_{1,zz}^2r^{2\mu}\bigg|_{r=0}^{r=R}\, \ud z + 2\mu(1-\mu) \int_\Omega \psi_{1,zz}^2r^{2\mu-2}\, \ud x,
        \label{4.10}
    \end{multline}
    where the first integral vanishes because $u_{,zz}|_{r=R}=0$ and Remark \ref{r2.4} yields that $u_{,zz}|_{r=0}=0$. Using (\ref{4.10}) in (\ref{4.9}) implies (\ref{4.5}). This ends the proof.
\end{proof}

\section{Proofs of theorems}

Combining Lemma \ref{lem3.1} with $k = 0$ and $k = 1$ along with Lemmas \ref{lem4.1} and \ref{lem4.2} we obtain 
\begin{multline*}
    \int_{-a}^a\|\psi_1-\psi_1^{(1)}(0)\|_{H_\mu^2(0,R)}^2\,\ud z+ \|\psi_{1,zz}\|_{L_{2,\mu}(\Omega)}^2 + \norm{\psi_{1,rz}}_{L_{2,\mu}(\Omega)}^2 \\
    + 2\mu(1-\mu)\int_\Omega \psi_{1,z}^2r^{2\mu-2}\, \ud x \leq c\|\omega_1\|_{L_{2,\mu}(\Omega)}^2,
\end{multline*}
and
\begin{multline*}
    \int_{-a}^a\|\psi_1-\psi_1^{(1)}(0)\|_{H_\mu^3(0,R)}^2\,\ud z+ \|\psi_{1,zzz}\|_{L_{2,\mu}(\Omega)}^2 + \norm{\psi_{1,rzz}}_{L_{2,\mu}(\Omega)}^2 \\
    + 2\mu(1 - \mu)\int_\Omega \psi_{1,zz}^2r^{2\mu-2}\, \ud x \leq c\|\omega_1\|_{H_\mu^1(\Omega)}^2,
\end{multline*}
where $\mu\in(0,1)$. This proves theorems \ref{t1} and \ref{t1.2}.

Lemmas \ref{lem3.7}, \ref{l2.3} and \ref{lem3.8} used with \eqref{eq1.14a} and \eqref{eq3.50u} for $k = 0$ yield
\begin{equation*}
    \int_{-a}^a\norm{\psi_1 - \psi_1^{(1)}(0) - \chi}_{H_0^2(0,R)}^2\,\ud z + \int_\Omega \left(\psi_{1,zz}^2 + \psi_{1,zr}^2\right)\, \ud x \leq c\norm{\omega_1}_{L_2(\Omega)}^2
\end{equation*}
and Lemmas \ref{l3.6}, \ref{2.3} and \ref{lem3.8} along with \eqref{eq1.14a} and \eqref{eq3.50u} for $k = 1$ give
\begin{multline*}
    \int_{-a}^a \norm{\psi_1 - \psi_1^{(1)}(0) - \eta}_{H_0^3(0,R)}^2\,\ud z+\int_\Omega (\psi_{1,zzz}^2 + \psi_{1,zzr}^2)\, \ud x + \norm{\psi_1}^2_{H^2(\Omega)}\\
    \leq c\norm{\omega_1}_{H^1(\Omega)}^2,
\end{multline*}
thus theorems \ref{t1.3} and \ref{t1.4} follow.

\bibliographystyle{elsarticle-num}
\bibliography{bibliography}

\begin{thebibliography}{10}
\expandafter\ifx\csname url\endcsname\relax
  \def\url#1{\texttt{#1}}\fi
\expandafter\ifx\csname urlprefix\endcsname\relax\def\urlprefix{URL }\fi
\expandafter\ifx\csname href\endcsname\relax
  \def\href#1#2{#2} \def\path#1{#1}\fi

\bibitem{Ladyzenskaja:1968aa}
O.~A. Lady\v{z}enskaja, Unique global solvability of the three-dimensional
  {C}auchy problem for the {N}avier-{S}tokes equations in the presence of axial
  symmetry, Zap. Nau\v{c}n. Sem. Leningrad. Otdel. Mat. Inst. Steklov. (LOMI) 7
  (1968) 155--177.

\bibitem{Ukhovskii:1968aa}
M.~R. Ukhovskii, V.~I. Iudovich, Axially symmetric flows of ideal and viscous
  fluids filling the whole space, J. Appl. Math. Mech. 32 (1968) 52--61.
\newblock \href {https://doi.org/10.1016/0021-8928(68)90147-0}
  {\path{doi:10.1016/0021-8928(68)90147-0}}.

\bibitem{Zhang:2014wk}
P.~Zhang, T.~Zhang, \href{https://doi.org/10.1093/imrn/rns232}{Global
  axisymmetric solutions to three-dimensional {N}avier-{S}tokes system}, Int.
  Math. Res. Not. IMRN~(3) (2014) 610--642.
\newblock \href {https://doi.org/10.1093/imrn/rns232}
  {\path{doi:10.1093/imrn/rns232}}.
\newline\urlprefix\url{https://doi.org/10.1093/imrn/rns232}

\bibitem{Wei:2016tn}
D.~Wei, \href{https://doi.org/10.1016/j.jmaa.2015.09.088}{Regularity criterion
  to the axially symmetric {N}avier-{S}tokes equations}, J. Math. Anal. Appl.
  435~(1) (2016) 402--413.
\newblock \href {https://doi.org/10.1016/j.jmaa.2015.09.088}
  {\path{doi:10.1016/j.jmaa.2015.09.088}}.
\newline\urlprefix\url{https://doi.org/10.1016/j.jmaa.2015.09.088}

\bibitem{Chen:2017we}
H.~Chen, D.~Fang, T.~Zhang,
  \href{https://doi.org/10.3934/dcds.2017081}{Regularity of 3{D} axisymmetric
  {N}avier-{S}tokes equations}, Discrete Contin. Dyn. Syst. 37~(4) (2017)
  1923--1939.
\newblock \href {https://doi.org/10.3934/dcds.2017081}
  {\path{doi:10.3934/dcds.2017081}}.
\newline\urlprefix\url{https://doi.org/10.3934/dcds.2017081}

\bibitem{Liu:2022tx}
Y.~Liu, \href{https://doi.org/10.1016/j.jde.2022.01.011}{Solving the
  axisymmetric {N}avier-{S}tokes equations in critical spaces ({I}): {T}he case
  with small swirl component}, J. Differential Equations 314 (2022) 287--315.
\newblock \href {https://doi.org/10.1016/j.jde.2022.01.011}
  {\path{doi:10.1016/j.jde.2022.01.011}}.
\newline\urlprefix\url{https://doi.org/10.1016/j.jde.2022.01.011}

\bibitem{NZ}
B.~Nowakowski, W.~Zaj{\k a}czkowski, Global regular axially-symmetric solutions
  to the {N}avier-{S}tokes equations. (2022).

\bibitem{Leonardi:1999uj}
S.~Leonardi, J.~M\'{a}lek, J.~Ne\v{c}as, M.~Pokorn\'{y},
  \href{https://doi.org/10.4171/ZAA/903}{On axially symmetric flows in {$\bold
  R^3$}}, Z. Anal. Anwendungen 18~(3) (1999) 639--649.
\newblock \href {https://doi.org/10.4171/ZAA/903} {\path{doi:10.4171/ZAA/903}}.
\newline\urlprefix\url{https://doi.org/10.4171/ZAA/903}

\bibitem{LW}
J.-G. Liu, W.-C. Wang, Characterization and regularity for axisymmetric
  solenoidal vector fields with application to {N}avier-{S}tokes equation, SIAM
  J. Math. Anal. 41~(5) (2009) 1825--1850.
\newblock \href {https://doi.org/10.1137/080739744}
  {\path{doi:10.1137/080739744}}.

\bibitem{NZ2}
B.~Nowakowski, W.~M. Zaj{\k a}czkowski, Stability of non-swirl axisymmetric
  solutions to the {N}avier-{S}tokes equations, TBDIn review (2022).

\bibitem{Ko67}
V.~Kondrat'ev, Boundary value problems for elliptic equations in domains with
  conical or angular points, Trudy Moskov. Mat. Ob{\v s}{\v c}. 16 (1967)
  209--292, {E}nglish translation in: Trans. Mosc. Math. Soc. 16, 227-313
  (1967).

\bibitem{Stein:1970wn}
E.~M. Stein, Singular integrals and differentiability properties of functions,
  Princeton Mathematical Series, No. 30, Princeton University Press, Princeton,
  N.J., 1970.

\end{thebibliography}

\end{document}